\documentclass[11pt]{article}

\usepackage[utf8]{inputenc}
\usepackage{amsmath,amssymb,amsthm}
\usepackage[margin=1in]{geometry}
\usepackage{graphicx}
\usepackage{hyperref}
\usepackage{pgfplots}
\usepgfplotslibrary{groupplots}
\pgfplotsset{compat=1.18}

\newtheorem{theorem}{Theorem}
\newtheorem{proposition}[theorem]{Proposition}
\newtheorem{assumption}[theorem]{Assumption}
\newtheorem{corollary}[theorem]{Corollary}
\theoremstyle{remark}
\newtheorem{remark}[theorem]{Remark}

\newcommand{\R}{\mathbb{R}}
\newcommand{\cA}{\mathcal{A}}
\newcommand{\cZ}{\mathcal{Z}}
\newcommand{\dist}{\operatorname{dist}}
\newcommand{\Id}{\operatorname{Id}}
\DeclareMathOperator{\diag}{diag}
\DeclareMathOperator{\arctanh}{arctanh}
\title{Global attractors and fast-slow reduction for\\
finite-state actor-critic mean dynamics}
\author{Vladyslav Prytula\\
zooplus SE\\
\texttt{vladyslav.prytula@proton.me}}
\date{}

\begin{document}
\maketitle

\begin{abstract}
When a learning algorithm reshapes the data distribution it trains on, the
long-run behavior depends on the joint evolution of the policy, the value
estimate, and the data distribution.
We study finite-state actor-critic mean dynamics on the enlarged phase space
$(\theta,w,\mu)$, where $\theta$ is the actor parameter, $w$ is an auxiliary
critic state, and $\mu$ is a state-law variable (the distribution over states
induced by the current policy). The state-law coordinate follows the exact
controlled-Markov equation $\delta \dot\mu = Q_\theta^*\mu$. Under a softmax
actor with box confinement (a smooth proxy for parameter clipping), a uniformly
coercive linear critic equation, and a Lipschitz generator family
$\theta \mapsto Q_\theta$, we prove that for each $\delta>0$ the resulting
autonomous semiflow possesses a compact global attractor. Under a uniform
exponential-mixing assumption, we prove that the invariant-law map
$\theta \mapsto \mu_\theta$ is Lipschitz and that the reduced invariant-law
system on $(\theta,w)$ is well posed. Under an additional pathwise
exponential-stability estimate for the non-autonomous fast state equation, we
show that the exact flow tracks the reduced flow on every finite time interval
up to the initial layer, and that the exact attractors converge upper
semicontinuously to the lifted reduced attractor as $\delta \to 0$. We also
give a concrete finite-state reference-state minorization condition implying
the pathwise hypothesis.  All results are formalized in Lean~4 without custom
axioms.
\end{abstract}

\section{Introduction}

In a recommendation system, the algorithm selects items based on a learned
policy, users respond according to the items they see, and the resulting
interaction data feeds back into the next policy update.  The policy, the value
estimate that guides it, and the distribution of user states all evolve
together: the policy reshapes the population, and the reshaped population
changes the gradient the learner sees.  This feedback loop---termed
performative in the recent literature \cite{PZMH20}---appears equally in search
ranking, adaptive routing, and other settings where a learning agent's
decisions influence the data it trains on.

In actor-critic models, the deterministic mean dynamics of the joint updates
therefore depend on the actor parameter $\theta$ (encoding the policy), the
critic state $w$ (a value-function estimate), and the state distribution $\mu$
(the population the current policy has shaped).  Most actor-critic analyses
reduce the long-run behavior to a policy ODE on $\theta$ alone and treat the
data distribution as a derived quantity, either by replacing it with a
prescribed relaxation map or by assuming it has already equilibrated.  Here we
keep all three as co-evolving state variables on an enlarged autonomous phase
space.  The distribution obeys the forward equation of the controlled Markov
chain,
\[
\delta \dot\mu = Q_\theta^*\mu,
\]
while the actor and critic are driven by a softmax actor field and a uniformly
coercive linear critic equation.

We prove three results for a finite-state model with explicit actor
confinement, critic coercivity, and a Lipschitz generator family.  First, for
each $\delta>0$ the resulting autonomous semiflow on $(\theta,w,\mu)$ possesses
a compact global attractor---a fixed bounded set that eventually captures every
trajectory regardless of initialization.  Second, under a uniform
exponential-mixing hypothesis, the frozen invariant-law map
$\theta \mapsto \mu_\theta$ is Lipschitz, which identifies the natural reduced
system on $(\theta,w)$ obtained by inserting $\mu = \mu_\theta$ into the exact
dynamics.  Third, under an additional pathwise exponential-stability estimate
for the non-autonomous fast state equation, the exact flow tracks the reduced
flow on every finite time interval away from the initial layer, and the exact
attractors converge upper semicontinuously to the lifted reduced attractor as
$\delta \to 0$.  We also give a concrete reference-state minorization criterion
that implies the pathwise hypothesis.

The attractor is the appropriate asymptotic object here because it describes
all possible long-run behaviors of the coupled system at once---multiple
equilibria, connecting orbits, and any other recurrent structure---without
requiring the analyst to identify individual stable states or to rule out more
complex dynamics in advance.  Its compactness and invariance give uniform
bounds that hold for every bounded set of initial conditions, and upper
semicontinuity under the singular limit $\delta \to 0$ means that the long-run
picture of the full three-variable system remains close to that of the simpler
two-variable reduction whenever the distribution equilibrates quickly.  In
practice, Markov chain mixing times are typically much shorter than the
timescale on which the policy changes, which is the regime where this reduction
applies.

Classical stochastic-approximation and actor-critic analyses relate learning
iterates to limiting ODEs or differential inclusions
\cite{Ben96,BM,BorAC,KT03,LiuCZ25}.  On the controlled-process side,
regularity of invariant laws under perturbations of a kernel or generator is
also standard \cite{Mit,MT,Norris,Sch,Yuksel24,Truquet20}.  The present note
combines these ingredients in a stationary finite-state setting and asks for a
compact asymptotic object for all bounded initial data of the exact
enlarged-state flow.  The additional contribution is that the invariant-law
graph can be compared quantitatively with the exact flow: under a pathwise
contraction estimate for the non-autonomous fast equation, the full and reduced
dynamics stay close on finite time intervals, and the corresponding compact
attractors converge as $\delta \to 0$.  The entire theorem package has been
formalized in Lean~4 with no \texttt{sorry}, \texttt{admit}, or
\texttt{axiom}; the proof scripts are available at
\url{https://codeberg.org/VladP/actor-crit-in-recsys}.

\section{Model}

Fix integers $N,K,d,m \geq 1$ and write
\[
S := \{1,\dots,N\},
\qquad
A := \{1,\dots,K\}.
\]
Let $R_\theta,\tau,\lambda_c,\delta$ be positive constants, let
$r_{i,a} \in \R$, $\psi_{i,a} \in \R^d$, and $\varphi_{i,a} \in \R^m$ for
$(i,a) \in S \times A$, and set
\[
\Theta := [-R_\theta,R_\theta]^d,
\qquad
\Delta_S := \Bigl\{\mu \in [0,\infty)^N : \sum_{i=1}^N \mu_i = 1\Bigr\}.
\]

The actor policy is the softmax law
\[
\pi_\theta(a \mid i)
:=
\frac{\exp(\psi_{i,a} \cdot \theta/\tau)}
{\sum_{b=1}^K \exp(\psi_{i,b} \cdot \theta/\tau)},
\]
and the current occupancy measure associated with $(\theta,\mu)$ is
\[
\nu_{\theta,\mu}(i,a) := \mu_i \pi_\theta(a \mid i).
\]

The functions $b$, $A$, and $\widetilde G$ arise from averaging the critic and
actor updates over the current occupancy measure $\nu_{\theta,\mu}$. The
vector $b(\theta,\mu)$ is the averaged reward-weighted feature,
$A(\theta,\mu)$ is the regularized feature covariance, and
$\widetilde G(\theta,w,\mu)$ is the softmax policy-gradient direction.

Define
\[
b(\theta,\mu)
:=
\sum_{i,a} \nu_{\theta,\mu}(i,a)\, r_{i,a}\, \varphi_{i,a},
\]
\[
A(\theta,\mu)
:=
\lambda_c I_m
+
\sum_{i,a} \nu_{\theta,\mu}(i,a)\, \varphi_{i,a}\varphi_{i,a}^\top,
\]
\[
\widetilde G(\theta,w,\mu)
:=
\sum_{i,a} \nu_{\theta,\mu}(i,a)
\bigl(r_{i,a} + \varphi_{i,a}\cdot w - \tau \log \pi_\theta(a \mid i)\bigr)
\nabla_\theta \log \pi_\theta(a \mid i),
\]
and
\[
D(\theta)
:=
\diag(R_\theta^2 - \theta_1^2,\dots,R_\theta^2 - \theta_d^2),
\qquad
G(\theta,w,\mu) := D(\theta)\widetilde G(\theta,w,\mu).
\]

The term $\lambda_c I_m$ is a ridge-type regularization of the critic and is
the source of the uniform coercivity used below. Without this term one would
need a separate feature-richness or persistent-excitation assumption to obtain
a uniform lower bound on the critic covariance.

The damping factor $D(\theta)$ confines the actor to the box $\Theta$. It
should be read as a smooth confinement surrogate for projected actor updates on
a compact parameter set. Strict projection would make the right-hand side
discontinuous, turning the actor equation into a Filippov differential
inclusion and invalidating both the Lipschitz semiflow property used in Theorem
\ref{thm:attractor} and the variation-of-constants argument in Theorem
\ref{thm:tracking}. Smooth confinement isolates the performative fast-slow
phenomena from sliding-mode boundary effects.

For the controlled-chain component, let $Q_\theta$ be a generator matrix on $S$ for each
$\theta \in \Theta$. We let $Q_\theta$ act on observables and $Q_\theta^*$
denote its transpose/adjoint acting on column probability vectors. We also
write
\[
\cZ := \Bigl\{\xi \in \R^N : \sum_{i=1}^N \xi_i = 0\Bigr\}.
\]
We use the $\ell^1$-norm on $\R^N$ for the state-law estimates.

\begin{assumption}\label{ass:gen}
For every $\theta \in \Theta$, the matrix $Q_\theta$ is a generator:
\[
q_{ij}(\theta) \geq 0 \quad \text{for } i \neq j,
\qquad
q_{ii}(\theta) = - \sum_{j \neq i} q_{ij}(\theta).
\]
The map $\theta \mapsto Q_\theta^*$ is Lipschitz from $\Theta$ into
$\mathcal{L}(\R^N,\R^N)$ in the induced $\ell^1$ operator norm: there exists
$L_Q > 0$ such that
\[
\|Q_\theta^* - Q_{\bar\theta}^*\|_{1 \to 1}
\leq
L_Q |\theta - \bar\theta|
\qquad
\text{for all } \theta,\bar\theta \in \Theta.
\]
\end{assumption}

Under Assumption \ref{ass:gen}, we consider the exact autonomous enlarged-state
system
\begin{equation}\label{eq:main}
\dot\theta = G(\theta,w,\mu),
\qquad
\dot w = b(\theta,\mu) - A(\theta,\mu)w,
\qquad
\delta \dot\mu = Q_\theta^* \mu
\end{equation}
The first equation updates the policy in a damped gradient direction, the
second is a linear regression for the critic with time-varying data, and the
third is the forward equation of the controlled chain on a fast timescale
$\delta^{-1}$. The state space is
\[
X := \Theta \times \R^m \times \Delta_S.
\]
We equip $X$ with the product metric
\[
d_X\bigl((\theta,w,\mu),(\bar\theta,\bar w,\bar\mu)\bigr)
:=
|\theta-\bar\theta| + |w-\bar w| + |\mu-\bar\mu|_1.
\]

\section{Global attractor}

We now show that the system \eqref{eq:main} is globally well posed and
dissipative. The strategy is classical: we first show that each component stays
in its natural domain (the actor parameters in the box $\Theta$, the
distribution on the probability simplex $\Delta_S$), then use the critic's
coercivity to prove that the critic state cannot escape to infinity. Together
these give a compact set $K$ that absorbs all trajectories in finite time.
Attractor existence then follows from standard finite-dimensional theory.

Set
\[
M_r := \max_{i,a} |r_{i,a}|,
\qquad
M_\varphi := \max_{i,a} |\varphi_{i,a}|,
\qquad
B_b := M_r M_\varphi,
\]
and define
\[
R_w := \max\Bigl\{1,\frac{2B_b}{\lambda_c}\Bigr\}.
\]

\begin{theorem}\label{thm:attractor}
Assume Assumption \ref{ass:gen}. For each $\delta>0$, system \eqref{eq:main}
generates a continuous semiflow $S_\delta(t)$ on $X$. The set
\[
K := \Theta \times \overline{B_{R_w}(0)} \times \Delta_S
\]
is compact, forward invariant, and absorbing for $S_\delta(t)$. Consequently,
$S_\delta(t)$ possesses a unique compact global attractor
$\cA_\delta = \omega(K)$.
\end{theorem}

\begin{proof}
Let $\Pi_\Theta : \R^d \to \Theta$ be the coordinatewise clipping map and
extend the generator family from $\Theta$ to $\R^d$ by
\[
\widetilde Q_\vartheta := Q_{\Pi_\Theta \vartheta},
\qquad
\vartheta \in \R^d.
\]
Since $\Pi_\Theta$ is $1$-Lipschitz and $\theta \mapsto Q_\theta^*$ is
Lipschitz on $\Theta$, the map $\vartheta \mapsto \widetilde Q_\vartheta^*$ is
locally Lipschitz on $\R^d$. Therefore the ambient vector field
\[
(\theta,w,\mu)
\mapsto
\left(
G(\theta,w,\mu),\,
b(\theta,\mu) - A(\theta,\mu)w,\,
\delta^{-1}\widetilde Q_\theta^* \mu
\right)
\]
is locally Lipschitz on $\R^d \times \R^m \times \R^N$, so each initial datum
in $X$ generates a unique local solution.

We first prove forward invariance of $\Theta$. For each coordinate,
\[
\dot\theta_j = (R_\theta^2 - \theta_j^2) h_j(\theta,w,\mu)
\]
for a continuous function $h_j$. If $|\theta_j(0)| < R_\theta$, then
$y_j := \arctanh(\theta_j/R_\theta)$ satisfies
\[
\dot y_j = R_\theta h_j(\theta,w,\mu),
\]
so $y_j$ cannot blow up in finite time and therefore $|\theta_j(t)| < R_\theta$
for all times of existence. If $\theta_j(0) = R_\theta$, set
$u_j(t):=R_\theta-\theta_j(t)$. Then
\[
\dot u_j
=
-\bigl(R_\theta + \theta_j(t)\bigr) h_j(\theta(t),w(t),\mu(t))\, u_j,
\]
so
\[
\frac{d}{dt}\left(
u_j(t)\exp\!\int_0^t \bigl(R_\theta + \theta_j(s)\bigr)
h_j(\theta(s),w(s),\mu(s))\,ds
\right) = 0.
\]
Since $u_j(0)=0$ and the exponential factor is strictly positive, $u_j(t)=0$
for all times of existence; hence $\theta_j(t)=R_\theta$. The case
$\theta_j(0)=-R_\theta$ is identical with $v_j(t):=R_\theta+\theta_j(t)$.
Hence $\theta(t) \in \Theta$ for all $t$.

For the distribution component, total mass is preserved because
\[
\frac{d}{dt}\sum_{i=1}^N \mu_i(t)
=
\delta^{-1}\sum_{i=1}^N (Q_{\theta(t)}^* \mu(t))_i
=
\delta^{-1}(Q_{\theta(t)}\mathbf{1}) \cdot \mu(t)
= 0.
\]
Suppose by contradiction that some component becomes negative, and define the
first exit time from the nonnegative cone by
\[
t_* := \inf\left\{t \geq 0 : \min_{1 \leq i \leq N}\mu_i(t) < 0\right\}.
\]
At time $t_*$ all components are nonnegative and at least one component, say
$\mu_i$, satisfies $\mu_i(t_*) = 0$. Since $Q_{\theta(t_*)}$ is a generator,
\[
\delta \dot\mu_i(t_*)
=
\sum_{j \neq i} q_{ji}(\theta(t_*)) \mu_j(t_*)
\geq 0.
\]
Thus $\mu_i$ cannot cross from nonnegative to negative values at $t_*$, a
contradiction. Since $\mu(0) \in \Delta_S$, every component stays nonnegative
and the total mass stays equal to $1$, so $\mu(t) \in \Delta_S$ for all times
of existence.

For the critic, the coercivity of $A(\theta,\mu)$ implies
\[
v \cdot A(\theta,\mu) v \geq \lambda_c |v|^2
\qquad
\text{for all } v \in \R^m.
\]
Since $|b(\theta,\mu)| \leq B_b$, we obtain
\[
\frac{d}{dt}|w|^2
=
2 w \cdot b(\theta,\mu) - 2 w \cdot A(\theta,\mu)w
\leq
-\lambda_c |w|^2 + \frac{B_b^2}{\lambda_c}.
\]
Gronwall's inequality yields
\[
|w(t)|^2
\leq
e^{-\lambda_c t}|w(0)|^2
+ \frac{B_b^2}{\lambda_c^2}\bigl(1-e^{-\lambda_c t}\bigr).
\]
No coordinate can therefore blow up in finite time, so the local solution
extends globally and defines a continuous semiflow $S_\delta(t)$ on $X$.

If $(\theta_0,w_0,\mu_0) \in K$, then the previous invariance arguments give
$\theta(t) \in \Theta$ and $\mu(t) \in \Delta_S$. Since
$R_w \geq B_b/\lambda_c$, the critic estimate shows that
$|w(t)| \leq R_w$ for all $t \geq 0$. Hence $K$ is forward invariant.

If $B \subset X$ is bounded and
$M_B := \sup\{|w| : (\theta,w,\mu) \in B\}$, choose $T_B \geq 0$ such that
\[
e^{-\lambda_c T_B} M_B^2
\leq
R_w^2 - \frac{B_b^2}{\lambda_c^2}.
\]
Then the same estimate gives $|w(t)| \leq R_w$ for every trajectory starting in
$B$ and every $t \geq T_B$. Therefore $S_\delta(t)B \subset K$ for $t \geq T_B$, so
$K$ absorbs every bounded subset of $X$.

We are now in the standard finite-dimensional dissipative setting: $K$ is
compact, forward invariant, and absorbing. It follows that
$\cA_\delta := \omega(K)$ is nonempty, compact, invariant, and attracts every
bounded subset of $X$; see, for example, Robinson \cite{Robinson} or Temam
\cite{Temam}. If $\cA_1$ and $\cA_2$ are two compact invariant sets attracting
all bounded subsets of $X$, then $\cA_1$ attracts $\cA_2$ and $\cA_2$ attracts
$\cA_1$; invariance then forces $\cA_1 = \cA_2$. Hence the global attractor is
unique.
\end{proof}

\begin{remark}\label{rem:boundary}
The confinement matrix $D(\theta)$ vanishes on $\partial\Theta$. A boundary
equilibrium of \eqref{eq:main} may therefore satisfy
$D(\theta_*)\widetilde G(\theta_*,w_*,\mu_*) = 0$ without satisfying
$\widetilde G(\theta_*,w_*,\mu_*) = 0$. The theorem proves dissipativity for the
boxed system exactly as stated; stronger claims about stationary points of an
undamped actor equation require a different confinement mechanism or an
interiority argument. In other words, the damping mechanism can create spurious
equilibria at the boundary of the parameter box where the policy gradient is
nonzero but the confinement factor kills it; these are artifacts of the boxing,
not genuine stationary points of the underlying actor-critic dynamics.
\end{remark}

\section{Stationary distribution under a fixed policy}

We now ask: if the policy parameter $\theta$ were held fixed, what stationary
distribution would the controlled chain converge to? The answer requires a
mixing condition ensuring that the chain forgets its initial state uniformly
over all fixed policies.

\begin{assumption}\label{ass:mix}
In addition to Assumption \ref{ass:gen}, there exist constants
$C_{\mathrm{mix}} \geq 1$ and $\gamma > 0$ such that
\[
|e^{tQ_\theta^*}\xi|_1
\leq
C_{\mathrm{mix}} e^{-\gamma t} |\xi|_1
\qquad
\text{for all } \theta \in \Theta,\ \xi \in \cZ,\ t \geq 0.
\]
\end{assumption}

Under Assumption \ref{ass:mix}, for each fixed $\theta \in \Theta$ the Markov
chain with generator $Q_\theta$ has a unique stationary distribution that
depends Lipschitz-continuously on $\theta$. Recent control and perturbation
results on policy-dependent invariant measures point in the same direction
\cite{Yuksel24,Truquet20}, but the finite-state resolvent argument below is
enough for the present note.

\begin{proposition}\label{prop:invlaw}
Under Assumption \ref{ass:mix}, for each $\theta \in \Theta$ there exists a
unique vector $\mu_\theta \in \Delta_S$ such that
\[
Q_\theta^* \mu_\theta = 0.
\]
Moreover, the map $\Theta \ni \theta \mapsto \mu_\theta \in \Delta_S$ is
Lipschitz and satisfies
\[
|\mu_\theta - \mu_{\bar\theta}|_1
\leq
\frac{C_{\mathrm{mix}}L_Q}{\gamma} |\theta - \bar\theta|
\qquad
\text{for all } \theta,\bar\theta \in \Theta.
\]
\end{proposition}

\begin{proof}
Fix $\theta \in \Theta$ and $\mu^0 \in \Delta_S$. Since $Q_\theta$ is a
generator, the semigroup $e^{tQ_\theta^*}$ is column-stochastic, so
$e^{tQ_\theta^*}\mu^0 \in \Delta_S$ for every $t \geq 0$. If $0 \leq t \leq s$,
then
\[
e^{sQ_\theta^*}\mu^0 - e^{tQ_\theta^*}\mu^0
=
e^{tQ_\theta^*}\bigl(e^{(s-t)Q_\theta^*}\mu^0 - \mu^0\bigr),
\]
and the difference in parentheses belongs to $\cZ$. Assumption
\ref{ass:mix} therefore gives
\[
|e^{sQ_\theta^*}\mu^0 - e^{tQ_\theta^*}\mu^0|_1
\leq
2 C_{\mathrm{mix}} e^{-\gamma t}.
\]
Hence $e^{tQ_\theta^*}\mu^0$ is Cauchy as $t \to \infty$, so it converges to a
limit $\mu_\theta \in \Delta_S$. The semigroup property gives
$e^{rQ_\theta^*}\mu_\theta = \mu_\theta$ for every $r \geq 0$, hence
$Q_\theta^* \mu_\theta = 0$.

If $\nu_\theta \in \Delta_S$ also satisfies $Q_\theta^*\nu_\theta = 0$, then
\[
e^{tQ_\theta^*}(\mu_\theta - \nu_\theta) = \mu_\theta - \nu_\theta
\qquad
\text{for all } t \geq 0.
\]
Because $\mu_\theta - \nu_\theta \in \cZ$, Assumption \ref{ass:mix} yields
\[
|\mu_\theta - \nu_\theta|_1
\leq
C_{\mathrm{mix}} e^{-\gamma t} |\mu_\theta - \nu_\theta|_1.
\]
Letting $t \to \infty$ gives $\mu_\theta = \nu_\theta$.

For the Lipschitz estimate, define the resolvent operator on $\cZ$ by
\[
\mathcal{R}_\theta \xi := -\int_0^\infty e^{tQ_\theta^*}\xi\,dt.
\]
The integral converges absolutely and
\[
|\mathcal{R}_\theta \xi|_1
\leq
\frac{C_{\mathrm{mix}}}{\gamma} |\xi|_1.
\]
A second integration-by-parts computation gives
$\mathcal{R}_\theta Q_\theta^* = \Id$ on $\cZ$. Subtraction of the
relations $Q_\theta^*\mu_\theta = 0$ and $Q_{\bar\theta}^*\mu_{\bar\theta} = 0$
gives
\[
\mu_\theta - \mu_{\bar\theta}
=
\mathcal{R}_\theta (Q_{\bar\theta}^* - Q_\theta^*) \mu_{\bar\theta}.
\]
Therefore
\[
|\mu_\theta - \mu_{\bar\theta}|_1
\leq
\frac{C_{\mathrm{mix}}}{\gamma}
\|Q_{\bar\theta}^* - Q_\theta^*\|_{1 \to 1}
|\mu_{\bar\theta}|_1
\leq
\frac{C_{\mathrm{mix}}L_Q}{\gamma} |\theta - \bar\theta|,
\]
because $|\mu_{\bar\theta}|_1 = 1$.
\end{proof}

\begin{remark}
Under Assumption \ref{ass:mix}, the exact system \eqref{eq:main} has the
natural reduced invariant-law system
\[
\dot{\bar\theta} = G(\bar\theta,\bar w,\mu_{\bar\theta}),
\qquad
\dot{\bar w}
=
b(\bar\theta,\mu_{\bar\theta}) - A(\bar\theta,\mu_{\bar\theta})\bar w
\]
on $Y := \Theta \times \R^m$. Proposition \ref{prop:invlaw} makes the vector
field locally Lipschitz, and the same box-invariance and critic-energy
estimates as in Theorem \ref{thm:attractor} show that
\[
K_Y := \Theta \times \overline{B_{R_w}(0)}
\]
is compact, forward invariant, and absorbing for the reduced semiflow
$\bar S(t)$ on $Y$. Hence $\bar S(t)$ possesses a compact global attractor
$\cA_0 = \omega(K_Y)$. Its lift
\[
\widehat{\cA}_0
:=
\left\{
(\theta,w,\mu_\theta) \in X : (\theta,w) \in \cA_0
\right\}
\]
is compact, and the lifted reduced semiflow
\[
\Gamma_0(t)(\theta_0,w_0,\mu_0)
:=
\bigl(\bar\theta(t),\bar w(t),\mu_{\bar\theta(t)}\bigr)
\]
has global attractor $\widehat{\cA}_0$.
\end{remark}

\section{Fast-slow reduction}

The previous section showed that each fixed policy has a unique stationary
distribution. In the coupled system, however, the policy is not fixed---it
evolves on a slow timescale while the distribution tries to track it on a fast
timescale. The question is whether the distribution actually stays close to the
moving target $\mu_{\theta(t)}$. This is a singular perturbation problem: as
$\delta \to 0$, the distribution equation speeds up and one expects $\mu$ to
approximately slave to the instantaneous stationary distribution. Making this
rigorous requires controlling the distribution dynamics along the actual
trajectory $\theta(t)$, not merely at frozen parameter values. The frozen
resolvent estimate from Proposition \ref{prop:invlaw} is not sufficient for
this, because along the coupled flow the fast equation is non-autonomous
\cite{Khalil,Verhulst}. The next assumption isolates the additional estimate needed.

\begin{assumption}\label{ass:path}
In addition to Assumption \ref{ass:mix}, there exist constants
$C_{\mathrm{path}} \geq 1$ and $\gamma_{\mathrm{path}} > 0$ such that, for every
measurable path $\vartheta : [0,\infty) \to \Theta$, the evolution family
$U_\vartheta^\delta(t,s) : \cZ \to \cZ$ of
\[
\delta \dot\xi(t) = Q_{\vartheta(t)}^* \xi(t),
\qquad
\xi(s) = \xi_s \in \cZ,
\]
satisfies
\[
|U_\vartheta^\delta(t,s)\xi|_1
\leq
C_{\mathrm{path}} e^{-\gamma_{\mathrm{path}}(t-s)/\delta} |\xi|_1
\qquad
\text{for all } \xi \in \cZ,\ 0 \leq s \leq t.
\]
\end{assumption}

\begin{theorem}\label{thm:tracking}
Assume Assumptions \ref{ass:mix} and \ref{ass:path}. Fix $T > 0$. Then there
exists a constant $C_T > 0$, independent of $\delta \in (0,1]$, such that the
following holds.

For $x_0 = (\theta_0,w_0,\mu_0) \in K$, write
\[
x^\delta(t)
:=
S_\delta(t)x_0
=
\bigl(\theta^\delta(t),w^\delta(t),\mu^\delta(t)\bigr),
\]
let
\[
\bigl(\bar\theta(t),\bar w(t)\bigr)
:=
\bar S(t)(\theta_0,w_0),
\]
and define the lifted reduced trajectory by
\[
\Gamma_0(t)x_0
:=
\bigl(\bar\theta(t),\bar w(t),\mu_{\bar\theta(t)}\bigr).
\]
Then, for every $t \in [0,T]$,
\[
\left|
\mu^\delta(t) - \mu_{\theta^\delta(t)}
\right|_1
\leq
C_T
\left(
e^{-\gamma_{\mathrm{path}} t/\delta}
\left|
\mu_0 - \mu_{\theta_0}
\right|_1
+
\delta
\right),
\]
and
\[
d_X\!\left(
S_\delta(t)x_0,
\Gamma_0(t)x_0
\right)
\leq
C_T
\left(
e^{-\gamma_{\mathrm{path}} t/\delta}
\left|
\mu_0 - \mu_{\theta_0}
\right|_1
+
\delta
\right).
\]
Consequently, for every fixed $t_* \in (0,T]$,
\[
\sup_{x_0 \in K}
\sup_{t_* \leq t \leq T}
d_X\!\left(
S_\delta(t)x_0,
\Gamma_0(t)x_0
\right)
\longrightarrow 0
\qquad
\text{as } \delta \to 0.
\]
\end{theorem}

\begin{proof}
Set
\[
M_G := \sup_{(\theta,w,\mu) \in K} |G(\theta,w,\mu)|,
\]
which is finite because $G$ is continuous on the compact set $K$. Since
$x^\delta(t)$ stays in $K$ by Theorem \ref{thm:attractor}, one has
$|\dot\theta^\delta(t)| \leq M_G$ for all $t \geq 0$.

Define
\[
\mu^\sharp(t) := \mu_{\theta^\delta(t)}.
\]
Proposition \ref{prop:invlaw} gives a Lipschitz constant $L_\mu$ for
$\theta \mapsto \mu_\theta$. Since $t \mapsto \theta^\delta(t)$ is absolutely
continuous, so is $t \mapsto \mu^\sharp(t)$, and
\[
|\dot\mu^\sharp(t)|_1 \leq L_\mu M_G
\]
for almost every $t \in [0,T]$. The state-layer error
\[
\eta(t) := \mu^\delta(t) - \mu^\sharp(t)
\]
belongs to $\cZ$ and satisfies
\[
\delta \dot\eta(t)
=
Q_{\theta^\delta(t)}^* \eta(t)
- \delta \dot\mu^\sharp(t)
\]
for almost every $t$. Variation of constants along the path
$\vartheta(t)=\theta^\delta(t)$ therefore gives
\[
\eta(t)
=
U_\vartheta^\delta(t,0)\eta(0)
- \int_0^t U_\vartheta^\delta(t,s)\dot\mu^\sharp(s)\,ds.
\]
Assumption \ref{ass:path} yields
\[
\begin{aligned}
|\eta(t)|_1
&\leq
C_{\mathrm{path}} e^{-\gamma_{\mathrm{path}} t/\delta} |\eta(0)|_1
+ C_{\mathrm{path}}
\int_0^t
e^{-\gamma_{\mathrm{path}}(t-s)/\delta}
|\dot\mu^\sharp(s)|_1\,ds \\
&\leq
C_{\mathrm{path}} e^{-\gamma_{\mathrm{path}} t/\delta}
\left|
\mu_0 - \mu_{\theta_0}
\right|_1
+
\frac{C_{\mathrm{path}}L_\mu M_G}{\gamma_{\mathrm{path}}}\delta.
\end{aligned}
\]
This is the first estimate after enlarging $C_T$.

Now define the actor-critic error
\[
E(t)
:=
\left|
\theta^\delta(t) - \bar\theta(t)
\right|
+
\left|
w^\delta(t) - \bar w(t)
\right|.
\]
Both trajectories stay in $K_Y$ by Theorem \ref{thm:attractor} and the remark
above, so the maps $G$, $b$, and $A$ are Lipschitz on the compact set
$K = K_Y \times \Delta_S$. Using the integral equations for the exact and
reduced systems, the Lipschitz continuity of $\theta \mapsto \mu_\theta$, and
the bound $|\bar w(t)| \leq R_w$, one obtains a constant $L_T > 0$ such that
\[
E(t)
\leq
L_T \int_0^t
\left(
E(s)
+
\left|
\mu^\delta(s) - \mu_{\theta^\delta(s)}
\right|_1
\right)\,ds
\]
for every $t \in [0,T]$. Gronwall's inequality and the estimate for $\eta$
therefore give
\[
\begin{aligned}
E(t)
&\leq
L_T e^{L_T T}
\int_0^t
\left|
\mu^\delta(s) - \mu_{\theta^\delta(s)}
\right|_1\,ds \\
&\leq
C_T
\left(
\delta
\left|
\mu_0 - \mu_{\theta_0}
\right|_1
+
\delta
\right).
\end{aligned}
\]
Since both $\mu_0$ and $\mu_{\theta_0}$ lie in $\Delta_S$, the factor
$|\mu_0-\mu_{\theta_0}|_1$ is bounded by $2$ on $K$, so after enlarging $C_T$
we may write simply
\[
E(t) \leq C_T \delta.
\]
Finally,
\[
\begin{aligned}
d_X\!\left(
S_\delta(t)x_0,
\Gamma_0(t)x_0
\right)
&\leq
E(t)
+
\left|
\mu^\delta(t) - \mu_{\theta^\delta(t)}
\right|_1
+
\left|
\mu_{\theta^\delta(t)} - \mu_{\bar\theta(t)}
\right|_1 \\
&\leq
(1+L_\mu)E(t)
+
\left|
\mu^\delta(t) - \mu_{\theta^\delta(t)}
\right|_1,
\end{aligned}
\]
so the second displayed estimate follows from the previous two bounds. The
uniform convergence on $[t_*,T]$ follows because
$|\mu_0-\mu_{\theta_0}|_1 \leq 2$ on $K$ and
$e^{-\gamma_{\mathrm{path}} t/\delta} \to 0$ uniformly for
$t \in [t_*,T]$ as $\delta \to 0$.
\end{proof}

\begin{corollary}\label{cor:usc}
Assume Assumptions \ref{ass:mix} and \ref{ass:path}. Then, for every
$\delta>0$, the exact semiflow $S_\delta(t)$ possesses a compact global
attractor $\cA_\delta \subset K$, and
\[
\dist_X\!\left(
\cA_\delta,
\widehat{\cA}_0
\right)
\longrightarrow 0
\qquad
\text{as } \delta \to 0,
\]
where
\[
\dist_X(B,C)
:=
\sup_{x \in B} \inf_{y \in C} d_X(x,y).
\]
\end{corollary}

\begin{proof}
Fix $\rho > 0$. Since $\widehat{\cA}_0$ is the global attractor of the lifted
reduced semiflow $\Gamma_0(t)$, there exists $T_\rho > 0$ such that
\[
\dist_X\!\left(
\Gamma_0(T_\rho)K,
\widehat{\cA}_0
\right)
< \frac{\rho}{2}.
\]
By Theorem \ref{thm:tracking}, one can choose $\delta$ small enough that
\[
\sup_{x_0 \in K}
d_X\!\left(
S_\delta(T_\rho)x_0,
\Gamma_0(T_\rho)x_0
\right)
< \frac{\rho}{2}.
\]
Since $\cA_\delta \subset K$ and $\cA_\delta = S_\delta(T_\rho)\cA_\delta$, we
obtain
\[
\begin{aligned}
\dist_X\!\left(
\cA_\delta,
\widehat{\cA}_0
\right)
&\leq
\dist_X\!\left(
S_\delta(T_\rho)K,
\widehat{\cA}_0
\right) \\
&\leq
\sup_{x_0 \in K}
d_X\!\left(
S_\delta(T_\rho)x_0,
\Gamma_0(T_\rho)x_0
\right)
+
\dist_X\!\left(
\Gamma_0(T_\rho)K,
\widehat{\cA}_0
\right)
< \rho.
\end{aligned}
\]
Since $\rho>0$ was arbitrary, the claim follows.
The argument is an instance of the general upper-semicontinuity framework for
attractors under perturbation; see \cite{HLR} and \cite{Robinson}.
\end{proof}

\begin{proposition}\label{prop:minorization}
Assume that there exist a distinguished state $k_* \in S$ and a constant
$\underline q_* > 0$ such that
\[
q_{ik_*}(\theta) \geq \underline q_*
\qquad
\text{for all } \theta \in \Theta \text{ and all } i \neq k_*.
\]
Define
\[
\overline q
:=
\sup_{\theta \in \Theta}
\max_{1 \leq i \leq N}
\sum_{j \neq i} q_{ij}(\theta),
\qquad
\tau_* := \underline q_*^{-1},
\qquad
\alpha_* := e^{-\overline q/\underline q_*}.
\]
Then, for every measurable path $\vartheta : [0,\infty) \to \Theta$, every
$\delta>0$, and every $s \geq 0$, each column of
$P_\vartheta^\delta(s+\delta\tau_*,s)$ dominates $\alpha_* e_{k_*}$, where
$P_\vartheta^\delta(t,s)$ denotes the evolution operator on $\R^N$ for
\[
\delta \dot p(t) = Q_{\vartheta(t)}^* p(t).
\]
Consequently, Assumption \ref{ass:path} holds. More precisely, if $N \geq 2$,
then it holds with
\[
C_{\mathrm{path}} := (1-\alpha_*)^{-1},
\qquad
\gamma_{\mathrm{path}}
:=
\underline q_*
\log\!\left(\frac{1}{1-\alpha_*}\right).
\]
If $N=1$, then $\cZ=\{0\}$ and Assumption \ref{ass:path} is trivial.
(This degenerate case, in which the contraction constants from the $N \geq 2$
argument become vacuous, was identified during the Lean formalization and
required a separate branch in the proof.)
\end{proposition}

\begin{proof}
Because $\Theta$ is compact and $\theta \mapsto Q_\theta$ is continuous,
$\overline q < \infty$. Fix a measurable path $\vartheta$, a scale $\delta>0$,
a time $s \geq 0$, and a state $i \in S$. Let $p^{(i)}(\cdot)$ solve
\[
\delta \dot p^{(i)}(t) = Q_{\vartheta(t)}^* p^{(i)}(t),
\qquad
p^{(i)}(s) = e_i.
\]
The same nonnegativity and total-mass argument as in Theorem
\ref{thm:attractor} gives $p^{(i)}(t) \in \Delta_S$ for all $t \geq s$.

If $i = k_*$, then
\[
\delta \frac{d}{dt} p_{k_*}^{(k_*)}(t)
\geq
- \overline q\, p_{k_*}^{(k_*)}(t),
\]
so Gronwall gives
\[
p_{k_*}^{(k_*)}(s+\delta\tau_*)
\geq
e^{-\overline q\tau_*}
=
\alpha_*.
\]
If $i \neq k_*$, then first
\[
\delta \frac{d}{dt} p_i^{(i)}(t)
\geq
- \overline q\, p_i^{(i)}(t),
\]
hence
\[
p_i^{(i)}(t)
\geq
e^{-\overline q (t-s)/\delta}
\qquad
\text{for all } t \geq s.
\]
Using $q_{ik_*}(\vartheta(t)) \geq \underline q_*$ and the nonnegativity of the
other components,
\[
\delta \frac{d}{dt} p_{k_*}^{(i)}(t)
\geq
\underline q_* e^{-\overline q (t-s)/\delta}
- \overline q\, p_{k_*}^{(i)}(t).
\]
Multiplication by the integrating factor
$e^{\overline q (t-s)/\delta}$ and integration from $s$ to
$s+\delta\tau_*$ yield
\[
p_{k_*}^{(i)}(s+\delta\tau_*)
\geq
e^{-\overline q/\underline q_*}
=
\alpha_*.
\]
Thus each column of $P_\vartheta^\delta(s+\delta\tau_*,s)$ dominates
$\alpha_* e_{k_*}$.

If $N=1$, then $\cZ=\{0\}$, so
$U_\vartheta^\delta(t,s)\xi = 0$ for every $\xi \in \cZ$ and Assumption
\ref{ass:path} is immediate, for example with
$C_{\mathrm{path}}=1$ and $\gamma_{\mathrm{path}}=1$.
Assume now that $N \geq 2$ and choose some $i_* \in S$ with $i_* \neq k_*$.
Then for every $\theta \in \Theta$,
\[
\overline q
\geq
\sum_{j \neq i_*} q_{i_*j}(\theta)
\geq
q_{i_*k_*}(\theta)
\geq
\underline q_*,
\]
hence
\[
\alpha_* = e^{-\overline q/\underline q_*} \leq e^{-1} < 1.
\]

Set $M_* := e_{k_*}\mathbf{1}^\top$ and
\[
R_{\vartheta,s}^\delta
:=
\frac{P_\vartheta^\delta(s+\delta\tau_*,s) - \alpha_* M_*}{1-\alpha_*}.
\]
The previous lower bound shows that $R_{\vartheta,s}^\delta$ has nonnegative
entries, and column sums are preserved because both
$P_\vartheta^\delta(s+\delta\tau_*,s)$ and $M_*$ are column-stochastic. Hence
$R_{\vartheta,s}^\delta$ is column-stochastic. If $\xi \in \cZ$ and
$\xi = \xi^+ - \xi^-$ is the usual positive-negative decomposition, then
$|\xi^+|_1 = |\xi^-|_1 = |\xi|_1/2$. For any column-stochastic matrix $R$,
\[
|R\xi|_1
\leq
|R\xi^+|_1 + |R\xi^-|_1
=
|\xi^+|_1 + |\xi^-|_1
=
|\xi|_1.
\]
Since $M_*\xi = 0$ on $\cZ$, this gives the block contraction
\[
\left|
U_\vartheta^\delta(s+\delta\tau_*,s)\xi
\right|_1
\leq
(1-\alpha_*) |\xi|_1.
\]
The same simplex argument implies that every propagator
$P_\vartheta^\delta(t,s)$ is column-stochastic, hence nonexpansive on $\cZ$ in
$|\cdot|_1$. For arbitrary $0 \leq s \leq t$, let
$n := \lfloor (t-s)/(\delta\tau_*) \rfloor$. Iterating the block contraction on
the $n$ full blocks and using nonexpansiveness on the remainder interval yield
\[
\left|
U_\vartheta^\delta(t,s)\xi
\right|_1
\leq
(1-\alpha_*)^n |\xi|_1
\leq
(1-\alpha_*)^{-1}
e^{-\gamma_{\mathrm{path}}(t-s)/\delta}
|\xi|_1.
\]
This is Assumption \ref{ass:path}.
\end{proof}

\section{A two-state example}

We illustrate Theorem \ref{thm:attractor} on the smallest system that exhibits
genuine three-way coupling among $\theta$, $w$, and $\mu$. The example is a
two-state, two-action MDP in which a learner chooses between two actions in
each of two states, with the transition rates between states depending on the
policy parameter. Take $N=K=2$,
$d=m=1$, and set
\[
R_\theta = \tau = \lambda_c = 1,
\qquad
\delta = 2.
\]
The actor features are antisymmetric,
$\psi_{1,1}=\psi_{2,2}=+1$, $\psi_{1,2}=\psi_{2,1}=-1$;
the critic features are $\varphi_{1,1}=\varphi_{2,2}=1$,
$\varphi_{1,2}=\varphi_{2,1}=0$;
and the rewards are $r_{1,1}=1$, $r_{2,2}=\tfrac{1}{2}$,
$r_{1,2}=r_{2,1}=0$.
The reward asymmetry ($r_{1,1}\neq r_{2,2}$) ensures that the distribution
$\mu$ enters the actor drift; with equal rewards the actor and critic
decouple from~$\mu$.

For the controlled chain, define $Q_\theta$ by the rates
$q_{12}(\theta) = 1 - \theta/2$ and $q_{21}(\theta) = 1 + \theta/2$,
both positive on $\Theta=[-1,1]$. The spectral gap is
$q_{12}+q_{21}=2$ uniformly, so Assumption \ref{ass:mix} holds with
$C_{\mathrm{mix}}=1$ and $\gamma=2$. The invariant law is
$\mu_\theta = \bigl(\tfrac{1}{2}+\tfrac{\theta}{4},\;
\tfrac{1}{2}-\tfrac{\theta}{4}\bigr)$.
The minorization criterion from Proposition \ref{prop:minorization} also holds:
choosing $k_*=1$, one has
\[
q_{21}(\theta) = 1 + \theta/2 \geq \frac{1}{2}
\qquad
\text{for all } \theta \in [-1,1].
\]
Thus Assumption \ref{ass:path} is automatic in this example.

Writing
\[
p := \frac{1}{1+e^{-2\theta}},
\qquad
q := 1-p,
\]
the exact system \eqref{eq:main} reduces to three scalar equations:
\begin{align*}
\dot\theta &= (1-\theta^2)\,2pq\,
\bigl(\tfrac{1}{2}\mu_1 + \tfrac{1}{2} + w - 2\theta\bigr), \\
\dot w &= p\,\tfrac{\mu_1+1}{2} - (1+p)\,w, \\
\dot\mu_1 &= \bigl(\tfrac{1}{2}+\tfrac{\theta}{4}\bigr) - \mu_1.
\end{align*}
The first equation drives $\theta$ toward the policy-gradient direction, damped
at the boundary; the second is an exponentially stable linear filter for the
value estimate; and the third relaxes $\mu_1$ toward the stationary probability
of being in state~$1$ under the current policy.
The standing constants are $B_b=1$ and $R_w=2$, giving the absorbing set
$K = [-1,1]\times[-2,2]\times[0,1]$.

The reduced invariant-law system from the remark after Proposition
\ref{prop:invlaw} becomes
\begin{align*}
\dot{\bar\theta}
&=
(1-\bar\theta^2)\,2\bar p\bar q\,
\bigl(\tfrac{3}{4}+\tfrac{\bar\theta}{8}+\bar w - 2\bar\theta\bigr), \\
\dot{\bar w}
&=
\bar p\,\bigl(\tfrac{3}{4}+\tfrac{\bar\theta}{8}\bigr) - (1+\bar p)\,\bar w,
\end{align*}
where $\bar p = (1+e^{-2\bar\theta})^{-1}$ and $\bar q = 1-\bar p$.

Figure \ref{fig:attractor} shows three trajectories integrated from different
initial data, two of which start outside $K$. All trajectories enter $K$ and
converge to a unique interior equilibrium
$(\theta^*,w^*,\mu_1^*)\approx(0.59,\,0.36,\,0.65)$.
For the plotted value $\delta=2$, the exact attractor is numerically a single
point. In general, $\cA_\delta$ may contain multiple equilibria and connecting
orbits; the theorem guarantees only that it is compact, invariant, and attracts
all bounded sets.

\begin{figure}[ht]
\centering
\includegraphics[width=\textwidth]{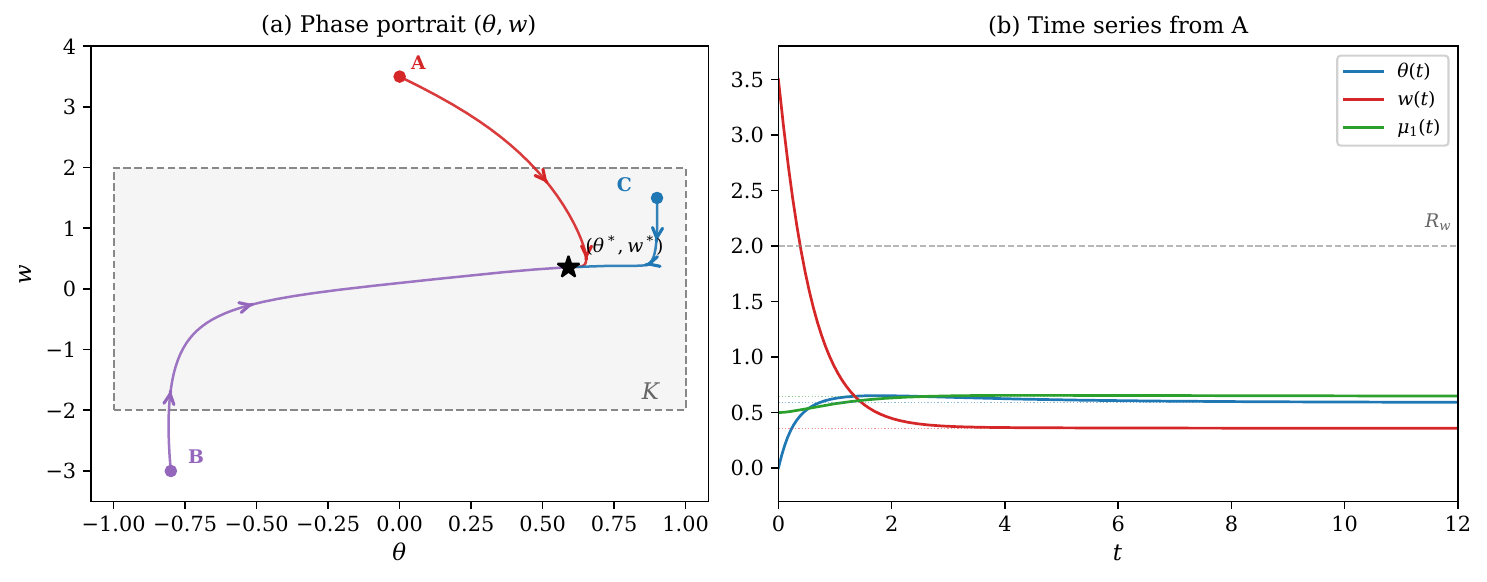}
\caption{Two-state example with asymmetric rewards.
(a)~Phase portrait projected onto $(\theta,w)$. The dashed rectangle is the
absorbing set~$K$; trajectories A and B start outside it. All three
trajectories converge to the interior equilibrium~(black star).
(b)~Time series of $\theta$, $w$, $\mu_1$ along trajectory~A.}
\label{fig:attractor}
\end{figure}

Because the equilibrium condition already imposes
$\mu_1 = \tfrac{1}{2}+\tfrac{\theta}{4}$, the equilibrium itself is independent
of $\delta$, so a direct attractor-convergence plot is not very informative in
this example. Figure \ref{fig:tracking} instead illustrates Theorem
\ref{thm:tracking}: for a fixed initial condition
$(\theta_0,w_0,\mu_{1,0})=(0.2,-1.4,0.05)\in K$, the state-law defect
$|\mu_1^\delta(t)-(\tfrac{1}{2}+\theta^\delta(t)/4)|$ exhibits the expected
initial layer, while the actor-critic discrepancy
$|\theta^\delta(t)-\bar\theta(t)|+|w^\delta(t)-\bar w(t)|$ is uniformly
$O(\delta)$ on the plotted interval.

\begin{figure}[ht]
\centering
\begin{tikzpicture}
\begin{groupplot}[
group style={group size=2 by 1, horizontal sep=1.6cm},
width=0.45\textwidth,
height=0.33\textwidth,
xlabel={$t$},
ylabel={error},
xmin=0,
xmax=8,
ymode=log,
grid=both,
major grid style={line width=0.2pt, draw=gray!35},
minor grid style={line width=0.1pt, draw=gray!20},
tick label style={font=\small},
label style={font=\small},
title style={font=\small},
legend style={font=\small, draw=none, fill=none}
]
\nextgroupplot[
title={(a) State-law defect}
]
\addplot[thick, color=red!75!black]
table[x=t, y=state_defect_1, col sep=comma]{tracking_data.csv};
\addplot[thick, color=orange!85!black]
table[x=t, y=state_defect_05, col sep=comma]{tracking_data.csv};
\addplot[thick, color=blue!75!black]
table[x=t, y=state_defect_02, col sep=comma]{tracking_data.csv};
\addplot[thick, color=green!55!black]
table[x=t, y=state_defect_01, col sep=comma]{tracking_data.csv};

\nextgroupplot[
title={(b) Actor-critic tracking error},
legend style={
font=\small,
draw=none,
fill=white,
fill opacity=0.85,
text opacity=1,
at={(0.03,0.05)},
anchor=south west
}
]
\addplot[thick, color=red!75!black]
table[x=t, y=tracking_error_1, col sep=comma]{tracking_data.csv};
\addlegendentry{$\delta=1$}
\addplot[thick, color=orange!85!black]
table[x=t, y=tracking_error_05, col sep=comma]{tracking_data.csv};
\addlegendentry{$\delta=0.5$}
\addplot[thick, color=blue!75!black]
table[x=t, y=tracking_error_02, col sep=comma]{tracking_data.csv};
\addlegendentry{$\delta=0.2$}
\addplot[thick, color=green!55!black]
table[x=t, y=tracking_error_01, col sep=comma]{tracking_data.csv};
\addlegendentry{$\delta=0.1$}
\end{groupplot}
\end{tikzpicture}
\caption{Finite-time reduction in the two-state example.
(a)~State-law defect
$|\mu_1^\delta(t)-(\tfrac{1}{2}+\theta^\delta(t)/4)|$
for $\delta \in \{1,0.5,0.2,0.1\}$.
(b)~Actor-critic tracking error
$|\theta^\delta(t)-\bar\theta(t)|+|w^\delta(t)-\bar w(t)|$
for the same initial $(\theta_0,w_0)$ and the same values of $\delta$. Both
panels use logarithmic vertical scale. Across both panels, red denotes
$\delta=1$, orange $\delta=0.5$, blue $\delta=0.2$, and green $\delta=0.1$.}
\label{fig:tracking}
\end{figure}
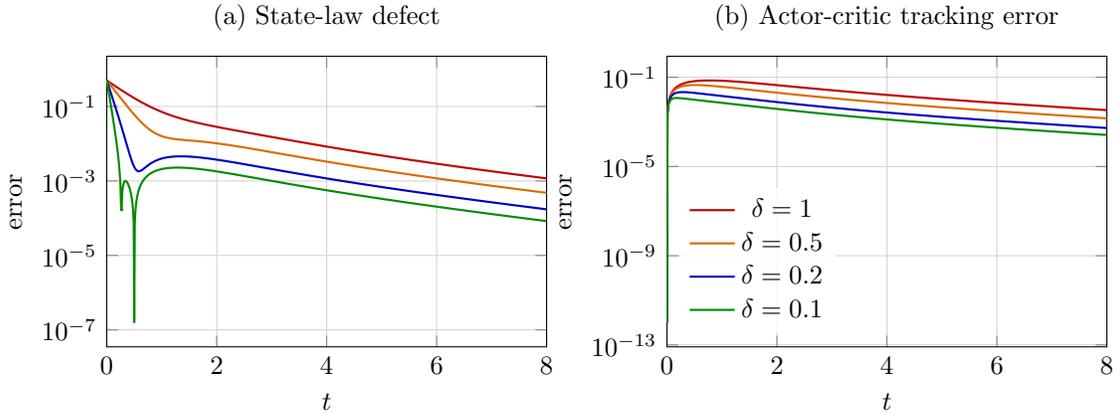

\clearpage

\section{Discussion}

The main message is that feedback between a learning algorithm and the
population it acts on can be analyzed as a dissipative dynamical system with a
compact attractor.  The attractor may contain multiple equilibria or connecting
orbits---it is a set whose internal structure depends on the model
parameters---and its existence, together with the upper-semicontinuity result,
gives a complete asymptotic picture: the long-run behavior of the full coupled
system, and its relationship to the simpler invariant-law reduction, are both
captured by compact invariant objects that deform continuously as
$\delta \to 0$.  In particular, upper semicontinuity is a one-sided guarantee:
as the population mixing time shrinks relative to the policy timescale, the
exact system cannot develop asymptotic behaviors---limit cycles, additional
equilibria, or other recurrent structures---that are absent from the reduced
invariant-law system.

Uniform exponential mixing enters first in the frozen invariant-law theorem,
which identifies the stationary occupancy and the reduced system.  To pass from
that frozen statement to an exact fast-slow comparison, one must control the
non-autonomous evolution family generated by $Q_{\theta(t)}^*$.  That extra
step is encoded here as Assumption~\ref{ass:path}, and
Proposition~\ref{prop:minorization} gives one explicit finite-state criterion
under which it holds.  In the Lean formalization, Assumption~\ref{ass:path} had
to be decomposed into three components---existence of a measurable evolution
family, a continuous-path stability core, and an agreement condition linking
the two---because the non-autonomous fast equation requires both measurability
(for the variation-of-constants integral) and continuity (for the Gronwall
estimates) in ways that the paper's single assumption leaves implicit.

A natural extension is the two-parameter limit
$(\varepsilon,\delta)\to(0,0)$ in which the critic equation
$\varepsilon\dot w = b - Aw$ is also fast; that requires an additional pathwise
stability estimate for the critic and reduces the slow dynamics to a single
actor equation on $\Theta$.  The present model is finite-state, boxed, and
deterministic at the mean-dynamics level; genuine external time dependence
(uniform or trajectory attractors) and stochastic perturbations bridging back
to the discrete-time iterates are logically separate extensions.  In recommendation and search systems, users fall into finitely many segments
and items into finitely many categories, so the finite-state model applies
directly.  The attractor then describes the long-run population shift caused by
the deployed policy.

\subsection*{Acknowledgments}
The Lean~4 formalization accompanying this paper was developed with assistance
from Claude (Anthropic).

\end{document}